\documentclass[11pt]{article}
\usepackage{amssymb}
\newtheorem{precor}{{\bf Corollary}}

\newenvironment{cor}{\begin{precor}{\hspace{-0.5
               em}{\bf.\ }}}{\end{precor}}
\newtheorem{prerem}{{\bf Remark}}

\newtheorem{precon}{{\bf Conjecture}}

\newtheorem{predefin}{{\bf Definition}}

\newenvironment{defin}[1]{\begin{predefin}{\hspace{-0.5
                   em}{\bf.\ }}{\rm
#1}\hfill{$\spadesuit$}}{\end{predefin}}
\newtheorem{preexm}{{\bf Example}}

\newtheorem{preappl}{{\bf Application}}

\newtheorem{prelem}{{\bf Lemma}}

\newenvironment{lem}{\begin{prelem}{\hspace{-0.5
               em}{\bf.\ }}}{\end{prelem}}
\newtheorem{preproof}{{\bf Proof.\ }}

\newenvironment{proof}[1]{\begin{preproof}{\rm
               #1}\hfill{$\blacksquare$}}{\end{preproof}}
\newtheorem{presproof}{{\bf Sketch of Proof.\ }}

\newtheorem{prethm}{{\bf Theorem}}

\newenvironment{thm}{\begin{prethm}{\hspace{-0.5
               em}{\bf.\ }}}{\end{prethm}}
\newtheorem{prealphthm}{{\bf Theorem}}

\newenvironment{alphthm}{\begin{prealphthm}{\hspace{-0.5
               em}{\bf.\ }}}{\end{prealphthm}}
\newtheorem{prealphlem}{{\bf  Lemma}}

\newtheorem{prepro}{{\bf Proposition}}

\newtheorem{preprb}{{\bf Problem}}

\newtheorem{prequ}{{\bf Question}}

\newenvironment{qu}{\begin{prequ}{\hspace{-0.5
               em}{\bf.\ }}}{\end{prequ}}
\def\conct[#1,#2]{\mbox {${#1} \leftrightarrow {#2}$}}
\def\dconct[#1,#2]{\mbox {${#1} \rightarrow {#2}$}}
\def\deg[#1,#2]{\mbox {$d_{_{#1}}(#2)$}}
\def\mindeg[#1]{\mbox {$\delta_{_{#1}}$}}
\def\maxdeg[#1]{\mbox {$\Delta_{_{#1}}$}}
\def\outdeg[#1,#2]{\mbox {$d_{_{#1}}^{^+}(#2)$}}
\def\minoutdeg[#1]{\mbox {$\delta_{_{#1}}^{^+}$}}
\def\maxoutdeg[#1]{\mbox {$\Delta_{_{#1}}^{^+}$}}
\def\indeg[#1,#2]{\mbox {$d_{_{#1}}^{^-}(#2)$}}
\def\minindeg[#1]{\mbox {$\delta_{_{#1}}^{^-}$}}
\def\maxindeg[#1]{\mbox {$\Delta_{_{#1}}^{^-}$}}
\def\isdef{\mbox {$\ \stackrel{\rm def}{=} \ $}}
\def\dre[#1,#2,#3]{\mbox {${\cal E}_{_{#3}}(#1,#2)$}}
\def\pdre[#1,#2,#3]{\mbox {${\cal P}_{_{#3}}(#1,#2)$}}
\def\var[#1,#2]{\mbox {${\rm Var}_{_{#1}}(#2)$}}
\def\ls[#1]{\mbox {$\xi^{^{#1}}$}}
\def\hom[#1,#2]{\mbox {${\rm Hom}({#1},{#2})$}}
\def\onvhom[#1,#2]{\mbox {${\rm Hom^{v}}(#1,#2)$}}
\def\onehom[#1,#2]{\mbox {${\rm Hom^{e}}(#1,#2)$}}
\def\core[#1]{\mbox {$#1^{^{\bullet}}$}}
\def\cay[#1,#2]{\mbox {${\rm Cay}({#1},{#2})$}}
\def\cays[#1,#2]{\mbox {${\rm Cay_{s}}({#1},{#2})$}}
\def\dirc[#1]{\mbox {$\stackrel{\rightarrow}{C}_{_{#1}}$}}
\def\cycl[#1]{\mbox {${\bf Z}_{_{#1}}$}}

\begin{document}
%\setcounter{page}{183}
%{\footnotesize
%\maketitle
\footnotetext[1]{$\ast$This research was in part supported by
a grant from IPM (No. 90050114).}
%\footnotetext[1]{Correspondence should be addressed to {\tt
%hhaji@sbu.ac.ir}.}
%\footnotetext[2]{$^\ast$ This paper is partially supported by
%Shahid Beheshti University.}
\begin{center}
{\Large \bf Secure Frameproof Code Through  Biclique Cover}\\
\vspace*{0.5cm}
{\bf Hossein Hajiabolhassan$^\ast$ and Farokhlagha Moazami$^\dag$}\\
{\it $^\ast$Department of Mathematical Sciences}\\
{\it Shahid Beheshti University, G.C.}\\
{\it P.O. Box {\rm 1983963113}, Tehran, Iran}\\
{\it School of Mathematics\\
Institute for Research in Fundamental Sciences {\rm (}IPM{\rm )}}\\
{\it P.O. Box {\rm 193955746}, Tehran, Iran}\\
{\tt hhaji@sbu.ac.ir}\\
{\it $^\dag$Department of Mathematics} \\
{\it Alzahra University}\\
{\it P.O. Box {\rm 1993891176}, Tehran, Iran}\\
{\tt f.moazami@alzahra.ac.ir}\\ \ \\
%%%%%%%%%%%%%%%%%%%%%%%%%%%%%%%%%
%%%%%%%%%%%%%%%%%%%%%%%%%%%%%%%%%
\end{center}
\begin{abstract} For a binary code $\Gamma$ of length $v$, a $v$-word $w$ produces by a set of codewords $\{w^1,\ldots,w^r\} \subseteq \Gamma$ if for all $i=1,\ldots,v$, we have $w_i\in \{w_i^1, \ldots, w_i^r\}$ . We call a code $r$-secure frameproof of size $t$ if $|\Gamma|=t$ and for any $v$-word that is produced by two sets $C_1$ and $C_2$ of size at most $r$ then the intersection of these sets is nonempty.
A $d$-biclique cover of size $v$ of a graph $G$ is a collection of $v$-complete bipartite subgraphs of $G$ such that each edge of $G$ belongs to at least $d$ of these complete bipartite subgraphs. In this paper, we show that for $t\geq 2r$, an $r$-secure frameproof code of size $t$ and length $v$ exists if and only if there exists a $1$-biclique cover of size $v$ for the Kneser graph ${\rm KG}(t,r)$ whose vertices are all $r$-subsets of a $t$-element set and  two $r$-subsets are adjacent if their intersection is empty. Then we investigate some connection between
the minimum size of $d$-biclique covers of Kneser graphs and cover-free families, where an $(r,w; d)$ cover-free family is a family of subsets of a finite set such that the intersection of any $r$ members of the family contains at least $d$ elements that are not in the union of any other $w$ members. Also, we present an upper bound for $1$-biclique covering number of Kneser graphs.
\begin{itemize}
\item[]{{\footnotesize {\bf Key words:}\ cover-free family, secure frameproof code, biclique cover, Hadamard matrix.}}
%%%%%%%%%%%%%%%%%%%%%%%%%%%%%%
%%%%%%%%%%%%%%%%%%%%%%%%%%%%%%
\item[]{ {\footnotesize {\bf Subject classification:} 05B40.}}
\end{itemize}
\end{abstract}
\section{Introduction} {\it Frameproof codes} were first introduced by Boneh and Shaw \cite{boneh}. Let $\Gamma\subseteq \{0,1\}^v$ and
$|\Gamma|=t$. $\Gamma$ is called a $(v,t)$-code and every
element of $\Gamma$ is said to be a code word. We write $w_i$ for the $i$th component of a word $w$.
Also, the incidence matrix of $ \Gamma$ is a $t\times v$
matrix whose rows are the codewords in $ \Gamma $. Suppose $C =
\{w^{(u_1)}, w^{(u_2)}, \ldots , w^{(u_d)}\}\subseteq \Gamma
\subseteq \{0,1\}^v$. For $i \in \{1, 2, \ldots, v\}$, the $i$th component is said undetectable for $C$ if
$$w^{(u_1)}_i= w_i^{(u_2)}=  \cdots = w_i^{(u_d)}.$$
Let $U(C)$ be the set of undetectable components for $C$. The
set
$$F(C) = \{x \in \{0, 1\}^v : x|_{U(C)} = w^{(u_i)}|_{U(C)} \ {\rm for \ all} \ w^{(u_i)}\in C \}$$
represents all possible $v$-tuples that could be
produced by the coalition $C$ by comparing the $d$ codewords they
jointly hold.
\begin{defin}{ An {\it r-frameproof code} is a subset $\Gamma\subseteq \{0,1\}^v$ such that for every $C\subseteq\Gamma$ where $|C|\leq r$, we have $ F(C)\cap\Gamma=C$.
}
\end{defin}
See \cite{blackburn, boneh, cff4, tardos2, tardos1} for more details about frameproof codes. The following theorem  was proved by Stinson and Wei ~\cite{frame}.
\begin{alphthm}{\rm \cite{frame}} {Suppose $\Gamma$ is an $r-FPC(v, b)$ with $b>2r - 1$. Suppose $D\subseteq\Gamma$, where
$|D| = 2r - 1$. Then there exists an unregistered word, say $maj(D) \in \{0,1\}^v$, such that
$maj(D) \in F(C)$ for any $C \subseteq D$ with $|C| = r$.
}
\end{alphthm}
In view of the aforementioned theorem, it is~not possible to identify a pirate user in an $r-FPC(v,b)$.
So they were considered a weaker condition and defined {\it secure frameproof codes} in which distributor is able to identify at least one pirate of the guilty coalitions.
\begin{defin}{Suppose that $\Gamma$ is a $(v, t)$-code. $\Gamma$ is said to be an $r$-secure frameproof
code if for any $C_1,  C_2 \subseteq \Gamma$  with $|C_1|\leq r$,
$|C_2|\leq r$, and $C_1 \cap C_2 = \varnothing$, we have $F(C_1)
\cap F(C_2) =\varnothing $.  Also, $\Gamma$ is termed  an $r-{\rm
SFPC}(v, t)$, for short. }
\end{defin}
Stinson and Wei in~\cite{frame}  studied the relationship between binary secure frameproof codes and combinatorial aspects. In this paper, we establish the relationship between this concept and biclique cover. By a {\it biclique} we mean a bipartite graph with vertex set $(X ,Y)$ such that every vertex in $X$ is
adjacent to every vertex in $Y$. Note that every empty graph is a biclique. A {\it $d$-biclique  cover} of a graph $G$ of size $s$ is a collection
of $s$ bicliques of $G$ such that each edge of $G$ is in at least $d$
of the bicliques. The $d$-{\it biclique covering number} of $G$, denoted by $bc_d(G)$, is defined to be the minimum number of $s$ such that there exists a $d$-biclique cover of size $s$  for the graph $G$.
\begin{defin}{Let $X$ be an $n$-set and ${\cal F}=\{B_1, \ldots, B_t\}$ be a family of subsets of $X$. ${\cal F}$ is called an $(r,w;d)$-cover-free family if for any two subsets $I, \ J \in [t]$ such that $|I|=r$, $|J|=w$, and $I\cap J=\varnothing$ the following condition holds
$$\displaystyle \bigcap_{i \in I} B_i \nsubseteq  \displaystyle \bigcup_{j \in J}B_j.$$
We denote it briefly by $(r,w)-CFF(n,t)$.}
\end{defin}
The minimum number of elements for which
there exists an $(r,w;d)-CFF$ with $t$ blocks is denoted by
$N((r,w;d),t)$. The incidence matrix of an $(r,w;d)-CFF$ is a $t\times n$ binary matrix $A$ such that $a_{ij}=1$ whenever $j \in B_i$ and $a_{ij}=0$ otherwise. As usual, we denote by $[t]$ the set $\{1, 2, \ldots, t\}$, and denote by ${[t] \choose r}$ the collection of all $r$-subsets of $[t]$. The graph $I_t(r,w)$ is a bipartite graph with the vertex set $({[t] \choose w},{[t]\choose r})$ which a $w$-subset is adjacent to an $r$-subset whenever their intersection is empty.
\begin{alphthm}{\rm ~\cite{haji}}\label{haji} For any positive integers $r$, $w$, $d$, and $t$, where $t\geq r+w$, we have
$$N((r,w;d),t)=bc_d(I_t(r,w)).$$
\end{alphthm}

 For abbreviation, let $bc(G)$ stand for $bc_1(G)$. The {\em Kneser graph} ${\rm KG}(t,r)$
is the graph with vertex set ${[t] \choose r}$, and $A$ is
adjacent to $B$ if $A \cap B = \varnothing$. Throughout this paper, we only consider
finite simple graphs. For a graph $G$, let $V(G)$ and $E(G)$ denote its vertex and edge
sets, respectively. A {\it homomorphism} from $G$ to $H$ is a map $\phi:
V(G)\longrightarrow V(H)$ such that adjacent vertices in $G$ are
mapped into adjacent vertices in $H$, i.e., $uv \in E(G)$ implies
$\phi(u)\phi(v) \in E(H)$. In addition, if any edge in $H$ is the
image of some edge in $G$, then $\phi$ is termed an onto-edge
homomorphism. In this paper, by $A^c$ we mean the complement of the set $A$.
In the next section, we show that for $t\geq 2r$, an $r$-secure frameproof code of size $t$ and length $v$ exists if and only if there exists a $1$-biclique cover of size $v$ for the Kneser graph ${\rm KG}(t,r)$. Also, we wish to investigate some connection between the $d$-biclique covering number of Kneser graphs and cover-free families. Finally, we present an upper bound for the biclique covering number of Kneser graphs.
%%%%%%%%%%%%%%%%%%%%%%%%%%%%%%%%%%%%%%%%%%%%%%%%%%%%%%%%%%%%%%%%%%%%%%%%%%%%%%%%%%%%%%%%%%%%%%%%%%%%%%%%%%%%%%%%%
%%%%%%%%%%%%%%%%%%%%%%%%%%%%%%%%%%%%%%%%%%%%%%%%%%%%%%%%%%%%%%%%%%%%%%%%%%%%%%%%%%%%%%%%%%%%%%%%%%%%%%%%%%%%%%%%%
\section{Secure Frameproof Codes }
For a subset $A_i$ of $[t]$, the {\it indicator vector} of $A_i$ is the vector $v_{A_i}=(v_1, \ldots, v_t)$, where $v_j=1$ if $j \in A_i$ and $v_j=0$ otherwise.
\begin{thm}\label{frameproof} Let $r$, $t$, and $v$ be positive integers, where $t\geq 2r$. An $r-{\rm SFPC}(v,t)$
exists if and only if there exists a biclique cover of size $v$ for the Kneser graph ${\rm KG}(t,r)$.
\end{thm}
\begin{proof}{Assume that $A$ is the incidence matrix of an $r-{\rm SFPC}(v,t)$.
Assign to the $j$th column of $A$, the set $A_j$ as follows
$$A_j\isdef \{i|\  1\leq i \leq t, a_{ij}=1\}.$$
 Now, for $1\leq j\leq v$, construct the
bicliques $G_j$ with vertex set $(X_j,Y_j)$, where the vertices
of $X_j$ are all $r$-subsets of $A_j$ and the vertices of $Y_j$
are all $r$-subsets of $A_j^c$ i.e., $[t]\setminus A_j$. It is easily seen that $G_j$, for
$1\leq  j\leq  v$, is a complete bipartite graph of ${\rm
KG}(t,r)$. Let $C_1C_2$ be an arbitrary edge of ${\rm KG}(t,r)$.
So $C_1, C_2 \subseteq [t]$, and $C_1\cap C_2 = \varnothing$.
Since $A$ is the incidence matrix of an $r-{\rm SFPC}(v,t)$, we
have $F(C_1)\cap F(C_2)= \varnothing$. This means that there
exists a bit position $i$ such that the $i$th bit of all code
words of $C_1$ is $c_i$, for some $c_i \in \{ 0, 1\}$, and also
the $i$th bit of all codewords of $C_2$ is $c_i+1  \ ({\rm mod} \
2)$. So there exists a column of $A$ such that all entries
corresponding to the rows of $C_1$ are equal to $1$ and all
entries corresponding to the rows of $C_2$ are equal to $0$, or
vice versa. Hence, $C_1C_2 \in E(G_i)$. Conversely, assume that
we have a biclique cover of size $v$ for the graph ${\rm
KG}(t,r)$. Our objective is to construct an $r$-SFPC. Label
graphs in this biclique cover with $G_1, \ldots, G_v$, where
$G_i$ has as its vertex set $(X_i, Y_i)$. Let $A_i$ be the union
of sets that lie in $X_i$. Consider the indicator vectors of
$A_i$ , for $1\leq  i\leq  v$, and construct the matrix $A$ whose
columns are these vectors. Assume that $C_1$ and $C_2$ are two
disjoint subsets of $[t]$ of size $r$, i.e, $ C_1C_2 \in E({\rm
KG}(t,r))$. Let $G_i$ be the complete bipartite graph that covers
the edge $C_1C_2$. Then in the $i$th column of
the matrix $A$ all entries corresponding to the rows of $C_1$ are
equal to $1$ and all entries corresponding to the rows of $C_2$
are equal to $0$, or vice versa. Consequently, $F(C_1)\cap
F(C_2)=\varnothing$.}
\end{proof}
%%%%%%%%%%%%%%%%%%%%%%%%%%%%%%%%%%%%%%%%%%%%%%%%%%%%%%%%%%%%%%%%%%%%%%%%%%%%%%%%%%%%%%%%%%%%%%%%%%%%%%%%%%%%%%%%%%%%%
%%%%%%%%%%%%%%%%%%%%%%%%%%%%%%%%%%%%%%%%%%%%%%%%%%%%%%%%%%%%%%%%%%%%%%%%%%%%%%%%%%%%%%%%%%%%%%%%%%%%%%%%%%%%%%%%%%%%%
 A $covering$ of a graph $G$ is a subset $K$ of $V(G)$ such that every edge of $G$ has at least one end in $K$. The number of
vertices in a minimum covering of $G$ is called the {\it covering  number} of $G$ and denoted by $\beta(G)$.
In \cite{frame}, Stinson, Trung, and Wei construct an $r-{\rm SFPC}(2{2r-1 \choose r-1}, 2r+1)$.
\begin{cor} {\rm \cite{frame}} For any integer $r\geq 0$, there exists an $r-{\rm SFPC}(2{2r-1 \choose r-1}, 2r+1).$
\end{cor}
\begin{proof}{Easily, one can check that the biclique covering number of a graph $G$
without $C_4$ as a subgraph is equal to the covering number of $G$. On the other hand ${\rm KG}(2r+1,r)$ does~not contain $C_4$ as a subgraph.
So $bc({\rm KG}(2r+1,r)))= \beta({\rm KG}(2r+1,r))$. Also, it is a
well-known fact that $\beta({\rm KG}(t,r))=\frac{t-r}{r}{t-1
\choose r-1}$. An easy computation confirms the assertion.}
\end{proof}
%%%%%%%%%%%%%%%%%%%%%%%%%%%%%%%%%%%%%%%%%%%%%%%%%%%%%%%%%%%%%%%%%%%%%%%%%%%%%%%%%%%%%%%%%%%%%%%%%%%%%%%%%%%%%%%%%
%%%%%%%%%%%%%%%%%%%%%%%%%%%%%%%%%%%%%%%%%%%%%%%%%%%%%%%%%%%%%%%%%%%%%%%%%%%%%%%%%%%%%%%%%%%%%%%%%%%%%%%%%%%%%%%%%
In the next theorem, we show the relationship between the
$d$-biclique cover of Kneser graphs and cover-free families.
\begin{thm}\label{kencff} For any positive integers $r$, $d$, and $t$, where $t\geq
2r$, it holds that
$$bc_{2d}({\rm KG}(t,r)) \leq N((r,r;d),t) \leq 2bc_d({\rm KG}(t,r)).$$
\end{thm}
\begin{proof}{
First, assume that we have an optimal $(r,r;d)-CFF(n,t)$, i.e.,
$n=N((r,r;d),t)$ with incidence matrix $A$. Assign to the $j$th
column of $A$ the set $A_j$ as follows
$$A_j\isdef \{i|\  1\leq i \leq t,\,\,\ a_{ij}=1\}.$$
Consider  the biclique $G_j$ with vertex set $(X_j,Y_j)$, where
the vertices of $X_j$ are all $r$-subsets of $A_j$ and the
vertices of $Y_j$ are all $r$-subsets of $A_j^c$. Also, two
vertices are adjacent if the subsets corresponding to these
vertices are disjoint. It is~not difficult to see that $G_j$'s,
for $1\leq j\leq  t$, form a $2d$-biclique cover of ${\rm
KG}(t,r)$.
So $bc_{2d}({\rm KG}(t,r)) \leq N((r,r;d),t)$.\\
Conversely, assume that we have a $d$-biclique cover of ${\rm
KG}(t,r)$. Label  graphs in this biclique cover with $G_1,
\ldots, G_l$, where $G_i$ has as its vertex set $(X_i, Y_i)$. Let
$A_i$ be the union of sets that lie in $X_i$ and $B_i$ be the
union of sets that lie in $Y_i$. Obviously, $A_i$ and $B_i$ are
disjoint. Consider the indicator vectors of  $A_i$'s and $B_i$'s,
for $i = 1, \ldots, l$.  Construct the matrix $A$ whose columns
are these vectors. Then $A$ is the
incidence matrix of an $(r,r;d)-CFF(2l,t)$. So $ N((r,r;d),t)
\leq 2bc_d({\rm KG}(t,r))$. }
\end{proof}
%%%%%%%%%%%%%%%%%%%%%%%%%%%%%%%%%%%%%%%%%%%%%%%%%%%%%%%%%%%%%%%%%%%%%%%%%%%%
%%%%%%%%%%%%%%%%%%%%%%%%%%%%%%%%%%%%%%%%%%%%%%%%%%%%%%%%%%%%%%%%%%%%%%%%%%%%%

By the aforementioned  results, it may be of interest to
find some bounds for the biclique covering number of Kneser
graphs.
%%%%%%%%%%%%%%%%%%%%%%%%%%%%%%%%%%%%%%%%%%%%%%%%%%%%%%%%%%%%%%%%%%%%%
\begin{thm}\label{GENKNE}
For any positive integers $d$, $r$, $s$, and $t$, where $t > 2r$
and $r>s$, we have
$$bc_d({\rm KG}(t,r))\geq bc_{m}({\rm KG}(t,s)),$$
where $m=N((r-s,r-s;d),t-2s)$.
\end{thm}
\begin{proof}{Let $\{ G_1, G_2, \ldots, G_l \}$ be an optimal $d$-biclique cover of ${\rm
KG}(t,r)$. Also, assume that $G_i$ has as its vertex set
$(X_i,Y_i)$. Let $A_i$ and $B_i$ be the union of sets that lie in
$X_i$ and $Y_i$, respectively. For any $1\leq i \leq l$, consider
the biclique $G'_i$, as a subgraph of ${\rm KG}(t,s)$,  with
vertex set $(X'_i,Y'_i)$, where $X'_i$ is the set of all $s$-subsets
of $A_i$ and $Y'_i$ is the set of all $s$-subsets of $B_i$. One
can check that $G'_i$'s cover all edges of ${\rm KG}(t,s)$.
Moreover, any edge $UV\in E({\rm KG}(t,s))$ is contained in at
least $m$-bicliques, where $m=N((r-s,r-s;d),t-2s)$. To see this,
consider the bipartite graph $I_{\{U,V\}}$ (as an induced subgraph
of ${\rm KG}(t,r)$) with vertex set $(X_U,Y_V)$, where
$$X_U=\{W|\  U\subseteq W\subseteq [t], W\cap V=\emptyset,
|W|=r\}$$
$$Y_V=\{W|\  V\subseteq W\subseteq [t], W\cap U=\emptyset,
|W|=r\}.$$ It is a simple matter to check that $I_{\{U,V\}}$  and
$I_{t-2s}(r-s,r-s)$ are isomorphic. Also, if $G_j$ covers any
edge of $I_{\{U,V\}}$, then $UV$ is contained in $G'_j$.
Consequently, by Theorem~\ref{haji} the assertion
follows.
%%%%%%%%%%%%%%%%%%%%%%%%%%%%%%%%%%%%%%%%%%%%%%%%%%%%%%%
}
\end{proof}

In view of the proof of Theorem \ref{GENKNE}, similarly, one can extend any
biclique of $I_t(r,w)$ to a biclique of $I_t(r-i,w-j)$.
Consequently, we have the following corollary.
%%%%%%%%%%%%%%%%%%%%%%%%%%%%%%
\begin{cor} Let $d$, $r$, $w$, and $t$ be positive integers, where $t\geq
r+w$. For any $1\leq i < r$ and $1\leq j < w$, we have
$N((r,w;d),t)\geq N((r-i,w-j;m),t)$, where
$m=N((i,j;d),t-r-w+i+j)$.
\end{cor}
We know that the image of a biclique under a graph homomorphism is
a biclique. This leads us to the following lemma.
%%%%%%%%%%%%%%%%%%%%%%%%%%%%%
\begin{lem}\label{do} Let $G$ and $H$ be two graphs and
$\phi : G \rightarrow H$ be an onto-edge homomorphism. Also,
assume that $d$ and $t$ are positive integers and for any edge $e \in E(H)$, $bc_d(\phi^{-1}(e)) \geq
t$. Then $bc_d(G)\geq bc_t(H).$
\end{lem}
\begin{proof}{Let $\{K_1,K_2, \ldots, K_l\}$ be an optimal $d$-biclique
cover of $G$. One can check that for any $0\leq i \leq
l$, $\phi(K_i)$ is a biclique and the family $\{\phi(K_1),\phi(K_2), \ldots,
\phi(K_l)\}$ is a $t$-biclique cover of $H$.}
\end{proof}

%%%%%%%%%%%%%%%%%%%%%%%%%%%%%%%%%%%%%%%%%%%%%%%%%%%%%%%%%%%%%%%%%%%%
%%%%%%%%%%%%%%%%%%%%%%%%%%%%%%%%%%%%%%%%%%%%%%%%%%%%%%%%%%%%%%%%%%%%%
\begin{thm}\label{KNE}
For any positive integers $t$ and $r$, where $t > 2r$, we have
$$bc_d({\rm KG}(t,r))\geq bc_{3d}({\rm KG}(t-2,r-1)).$$
\end{thm}
\begin{proof}{
First, we present an onto-edge homomorphism $\phi$ from ${\rm
KG}(t,r)$ to ${\rm KG}(t-2,r-1)$. To see this, for every vertex
$A$ of ${\rm KG}(t,r)$, define $\phi(A):=A'$ as follows. If $A$
does~not contain both $t$ and $t-1$, then define $A' := A\setminus
\{maxA\}$. Otherwise, set $A':= \{x\}\cup A \setminus \{t, t -
1\}$, where $x$ is the maximum element absent from $A$. It
is simple to check that the subgraph induced by the inverse image
of any edge of ${\rm KG}(t-2,r-1)$ contains an induced cycle of
size six or an induced matching of size three. Hence, in view of
Lemma \ref{do}, if $\{K_1, \ldots, K_l\}$ is a $d$-biclique cover
of ${\rm KG}(t,r)$, then $\{\phi(K_1), \ldots, \phi(K_l)\}$ is a
$3d$- biclique cover of ${\rm KG}(t-2,r-1)$.}
\end{proof}
%%%%%%%%%%%%%%%%%%%%%%%%%%%%%%%%%%%%%%%%%%%%%%%%%%%%%%%%%%%%%%%%%%%%%%%%%%%%%%%%%%%%%%%%%%%%%%%%%%%%%%%%%%%%%
%%%%%%%%%%%%%%%%%%%%%%%%%%%%%%%%%%%%%%%%%%%%%%%%%%%%%%%%%%%%%%%%%%%%%%%%%%%%%%%%%%%%%%%%%%%%%%%%%%%%%%%%%%%%%
The aforementioned results motivate us to consider
the following question.
\begin{qu}
Let $d$, $r$, and $t$ be positive integers, where $t > 2r$. What
is the exact value of $bc_d({\rm KG}(t,r))${\rm ?}
\end{qu}
%%%%%%%%%%%%%%%%%%%%%%%%%%%%%%%%%%%%%%%%%%%%%%%%%%%%%%%%%%%%%%%%%%%%%%%%%%%%%%%%%%%%%%%%%%%%%%%%%%%%%%%%%%%%%%
%%%%%%%%%%%%%%%%%%%%%%%%%%%%%%%%%%%%%%%%%%%%%%%%%%%%%%%%%%%%%%%%%%%%%%%%%%%%%%%%%%%%%%%%%%%%%%%%%%%%%%%%%%%%%%
An $n \times n$ matrix $H$ with entries
$+1$ and $-1$  is called a {\it Hadamard matrix} of {\it order} $n$  whenever $HH^t = nI.$
It is not difficult to see that any two columns of $H$ are also orthogonal. If we permute rows or columns or if we multiply some rows or columns by $-1$ then this property
does not change. Two such Hadamard matrices are
called {\it equivalent}. For a given Hadamard matrix, we can find an
equivalent one for which the first row and the first column consist
entirely of $+1$'s. Such a Hadamard matrix is called {\it normalized}.
We will denote by $K_{m,m}^-$ the complete bipartite graph with a perfect matching removed. Obviously, $K_{m,m}^-$ is isomorphic to $I_m(1,1)$.
\begin{thm}Let $d$ be a positive integer such that there exists a Hadamard matrix of order $4d$, then
\begin{description}
\item[1.]  $bc_{2d}(K_{8d})=4d,$
\item[2.] $N((1,1;d),8d-2)=bc_{d}(K^-_{8d-2,8d-2})=4d$.
\end{description}

\end{thm}
\begin{proof}{Let $H=[h_{ij}]$ be a Hadamard matrix of order
$4d$. Suppose that $K_{8d}$ has $\{u_1, \ldots, u_{4d}, v_1, \ldots , v_{4d}\}$ as its vertex set. For the $j$th column of $H$, two sets $X_j$ and $Y_j$ are defined as follows
$$ X_j:=\{ u_i | h_{ij}=+1 \} \cup \{ v_i | h_{ij}=-1 \} \,\,\,  \& \,\,\ Y_j:=\{u_i| h_{ij}=-1 \}\cup \{v_i| h_{ij}=+1 \}.$$
By constructing a bipartite graph $G_j$ with vertex set $(X_j, Y_j)$ indeed we assign a biclique to
each column. It is well-known that for any two rows of a Hadamard matrix, the number of columns for which corresponding entries in these rows are different in sign, are equal to $2d$. So, for $i\neq j$ the edges $u_iu_j$, $v_iv_j$ and $u_iv_j$ of the graph $K_{8d}$ are covered by $2d$ bicliques. Finally, consider the edge $u_iv_i$, then there exist $4d$ bicliques that cover it. According to the above argument every edge is covered at least $2d$ times, so $bc_{2d}(K_{8d})\leq 4d$. On the other hand, for every graph $G$ we have $\frac{|E(G)|}{B(G)}\leq  \frac{bc_d(G)}{d}$, therefore $$ 4d-\frac{1}{2} \leq bc_{2d}(K_{8d}).$$
Since $bc_{2d}(K_{8d})$ is an integer, we have $ 4d \leq bc_{2d}(K_{8d})$ which completes the proof.
For the proof of the second part, assume that $H$ is a normalized Hadamard matrix of order $4d$. Delete the first row of $H$ and denote it by $H'=[h'_{ij}]$. Also, assume that $K^-_{8d-2,8d-2}$ has $(X,Y)$ as its vertex set where $X=\{u_1, \ldots, u_{4d-1}, v_1, \ldots, v_{4d-1} \}$,  $Y=\{u'_1, \ldots, u'_{4d-1}, v'_1, \ldots, v'_{4d-1}\}$ and $u_iu_i', v_iv_i'\not \in E(K^-_{8d-2,8d-2})$.
 Assign to the $j$th column of $H'$, two sets $X_j$ and $Y_j$ as follows
$$ X_j:=\{ u_i | h'_{ij}=+1 \} \cup \{v_i | h'_{ij}=-1 \} \quad \&  \quad Y_j:=\{u'_i | h'_{ij}=-1 \}\cup \{v'_i|h'_{ij}=+1\}.$$
By the same argument in the first part of the proof and using the well-known fact that in $H'$ every two distinct rows $i,j$ have exactly $d$ columns that the corresponding entries are $+1$ and $-1$ in the rows $i$ and $j$, respectively, one can see that every edge is covered at least $d$ times. So $bc_d(K^-_{8d-2,8d-2})\leq 4d$. On the other hand $4d -\frac{2d}{4d-1}\leq bc_d(K^-_{8d-2,8d-2})$, and $\frac{2d}{4d-1}< 1$. Therefore $4d \leq bc_d(K^-_{8d-2,8d-2})$ which establishes the second part.
}
\end{proof}
%%%%%%%%%%%%%%%%%%%%%%%%%%%%%%%%%%%%%%%%%%%%%%%%%%%%%%%%%%%%%%%%%%%%%%%%%%%%%%%%%%%%%%%%%%%%%%%%%%%%%%%%%%%%%%
%%%%%%%%%%%%%%%%%%%%%%%%%%%%%%%%%%%%%%%%%%%%%%%%%%%%%%%%%%%%%%%%%%%%%%%%%%%%%%%%%%%%%%%%%%%%%%%%%%%%%%%%%%%%%%
Stinson {\it et al} \cite{frame}, using the probabilistic method, obtain an upper bound for $SFPC$.
In the next theorem, we present a slight improvement of this upper bound.
 %%%%%%%%%%%%%%%%%%%%%%%%%%%%%%%%%%%%%%%%%%%%%%%%%%%%%%%%%%%%%%%%%%%%%%%%%%%%%%%%%%%%%%%%%%%%%%%%%%%%%%%%%%%%%%
 %%%%%%%%%%%%%%%%%%%%%%%%%%%%%%%%%%%%%%%%%%%%%%%%%%%%%%%%%%%%%%%%%%%%%%%%%%%%%%%%%%%%%%%%%%%%%%%%%%%%%%%%%%%%%%
\begin{thm}\label{prob2} Let $r$ and $t$ be positive integers. If $t$ is sufficiently large respect to $r$ then there exists an $r-SFPC(v,t)$ where
$$v\leq \frac{{t \choose \lceil\frac{t}{2}\rceil}}{2{t-2r \choose \lceil\frac{t}{2}\rceil-r}}
(1+\ln({\lceil\frac{t}{2}\rceil \choose r}{\lfloor\frac{t}{2}\rfloor\choose r})).$$
\end{thm}
\begin{proof}{We show that if $v\geq \lfloor\frac{{t \choose \lceil\frac{t}{2}\rceil}}{2{t-2r \choose \lceil\frac{t}{2}\rceil-r}}(1+\ln({\lceil\frac{t}{2}\rceil \choose r}{\lfloor\frac{t}{2}\rfloor\choose r}))\rfloor$ then there exists a biclique cover of size $v$ for the Kneser graph ${\rm KG}(t,r)$. Let ${\cal A}$ be ${[t] \choose \lceil \frac{t}{2} \rceil}$. For every member of ${\cal A}$, say $A_i$, we can construct the biclique $G_i$ with vertex set $(X_i,Y_i)$, where the vertices of $X_i$ are all $r$-subsets of $A_i$ and the vertices of $Y_i$
are all $r$-subsets of $A_i^c$. We define ${\cal B}$ to be the collection contains all of these bicliques. Let $p\in \ [0,1]$ be arbitrary, later, we specify an optimized value for $p$. Let us
pick, randomly and independently, each biclique of ${\cal B}$ with
probability $p$ and ${\cal F}$ be the random set of all bicliques
picked and let $Y_{\cal F}$ be the set of all edges $AB$ of the graph ${\rm KG}(t,r)$ which are~not covered by the set ${\cal F}$. The expected value of $|{\cal F}|$ is clearly ${t \choose \lceil\frac{t}{2}\rceil}p$.
For every edge $AB$, $ pr(AB \in Y_{\cal F}) = (1-p)^l$ where $l=2{t-2r \choose \lceil\frac{t}{2}\rceil-r} $. So the expected value of the $|{\cal F}|+ |Y_{\cal F}| $ is at most
$${t \choose \lceil\frac{t}{2}\rceil}p + {1 \over 2}{t\choose r}{t-r \choose r}(1-p)^{2{t-2r \choose \lceil\frac{t}{2}\rceil-r}}.$$ If we set ${\cal F}' = {\cal F} \bigcup  Y_{\cal F}$, then clearly all edges of the graph ${\rm KG}(t,r)$ are covered by ${\cal F}'$. So we want to estimate $p$ such that $|{\cal F}'|$ is 
minimum. For convenient, we bound $1-p \leq e^{-p}$ to obtain $$ E(|{\cal F}|+ |Y_{\cal F}|) \leq {t \choose \lceil\frac{t}{2}\rceil}p + {1\over2}{t\choose r}{t-r \choose r}e^{-2{t-2r \choose \lceil\frac{t}{2}\rceil-r}p}. $$  The right hand side is minimized
at $p=\frac{\ln\alpha}{\beta}$, which $ \alpha = {\lceil\frac{t}{2}\rceil \choose r}{\lfloor\frac{t}{2}\rfloor\choose r}$   and $\beta = 2{t-2r \choose \lceil\frac{t}{2}\rceil-r}$ where $p \in [0,1]$ if $t$ is
sufficiently large respect to $r$. So we have an $r-SFPC(v,t)$ that
$$v\leq \frac{{t \choose \lceil\frac{t}{2}\rceil}}{2{t-2r \choose \lceil\frac{t}{2}\rceil-r}}
(1+\ln({\lceil\frac{t}{2}\rceil \choose r}{\lfloor\frac{t}{2}\rfloor\choose r})).$$}
\end{proof}
%%%%%%%%%%%%%%%%%%%%%%%%%%%%%%%%%%%%%%%%%%%%%%%%%%%%%%%%%%%%%%%%%%%%%%%%%%%%%%%%%%%%%%%%%%%%%%%%%%%%%%%%%%%%%%%%%
%%%%%%%%%%%%%%%%%%%%%%%%%%%%%%%%%%%%%%%%%%%%%%%%%%%%%%%%%%%%%%%%%%%%%%%%%%%%%%%%%%%%%%%%%%%%%%%%%%%%%%%%%%%%%%%%%
%reference------------------------------------------------------------------------

%\bibliographystyle{plain}

%\bibliography{bibref}

\def\cprime{$'$}

\end{document}